\documentclass{amsart}
\usepackage{amsfonts}
\usepackage{graphicx}
\usepackage{amscd}
\usepackage{amsmath}
\usepackage{amssymb}

\makeatletter
\@namedef{subjclassname@2010}{%
  \textup{2010} Mathematics Subject Classification}
\makeatother

\theoremstyle{plain}

\newtheorem{corollary}{\bf Corollary}
\newtheorem{definition}{\bf Definition}
\newtheorem{lemma}{\bf Lemma}

\newtheorem{proposition}{\bf Proposition}
\newtheorem{remark}{Remark}

\newtheorem{theorem}{\bf Theorem}

\theoremstyle{definition}

\numberwithin{equation}{section}

\title[Critical metrics of the volume functional]{On the volume functional of compact manifolds with boundary with harmonic Weyl tensor}

\author{H. Baltazar}
\author{R. Batista}
\author{K. Bezerra}

\address[H. Baltazar]{Departamento de Matem\'{a}tica, Universidade Federal do Piau\'{\i}\\
64049-550 Te\-re\-si\-na, Piau\'{\i}, Brazil.}
\email{halyson@ufpi.edu.br}

\address[R. Batista]{Departamento de Matem\'{a}tica, Universidade Federal do Piau\'{\i}\\
64049-550 Te\-re\-si\-na, Piau\'{\i}, Brazil.}
\email{rmarcolino@ufpi.edu.br}

\address[K. Bezerra]{Departamento de Matem\'{a}tica, Universidade Federal do Piau\'{\i}\\
64049-550 Te\-re\-si\-na, Piau\'{\i}, Brazil.}
\email{kelton@ufpi.edu.br}

\thanks{R. Batista was partially supported by CNPq/Brazil}
\subjclass[2010]{Primary 53C25, 53C20, 53C21; Secondary 53C65}
\keywords{Volume functional; critical metrics; Bach-flat metrics; harmonic Weyl tensor}
\date{July 28, 2017}

\begin{document}

\newcommand{\spacing}[1]{\renewcommand{\baselinestretch}{#1}\large\normalsize}
\spacing{1.2}

\begin{abstract}
One of the main aims of this article is to give the complete classification of critical metrics of the volume functional on a compact manifold $M$ with boundary $\partial M$ and with harmonic Weyl tensor, which improves the corresponding classification for complete locally conformally flat case, due to Miao and Tam~\cite{miaotamTAMS}. In particular, we prove that a critical metric with harmonic Weyl tensor on a simply connected compact manifold with boundary isometric to a standard sphere $\mathbb{S}^{n-1}$ must be isometric to a geodesic ball in a simply connected space form $\Bbb{R}^n,$ $\Bbb{H}^n$ and $\Bbb{S}^n.$ In order to achieve our goal, firstly we shall conclude the classification of such critical metrics under the Bach-flat assumption and then we will prove that both geometric conditions are indeed equivalent.
\end{abstract}

\maketitle

\section{Introduction}
\label{intro}
The study of the critical metrics constitutes a classical and fruitful theme in the theory of geometric analysis. Into this branch much attention has been given to study the critical points of the volume functional on the space of constant scalar curvature metrics with a prescribed boundary metric. In 2009, Miao and Tam \cite{miaotam} investigated variational properties of the volume functional constrained to the space of metrics of constant scalar curvature on a given compact manifold with prescribed boundary. In this context, they established sufficient conditions for a metric to be a critical point. More precisely, let $M^{n}$ ($n\geq3$) be a connected, compact $n$-dimensional manifold with smooth boundary $\partial M$ (possibly disconnected) which is endowed with a fixed metric $\gamma,$ and let $\mathcal{M}^{R}_{\gamma}$ be the subset of Riemannian metrics $g$ with constant scalar curvature $R$ and such that $g|_{\partial M}=\gamma.$ The authors proved that, if the first Dirichlet eigenvalue of $(n-1)\Delta_{g}+R$ on $M$ is positive, then $g$ is a critical point of the volume functional in $\mathcal{M}^{R}_{\gamma}$ if and only if there is a smooth function $f$ on $M$ such that  $f|_{\partial M}=0$ and satisfies the follows equation
\begin{equation}\label{eqMiaoTam1}
-(\Delta f)g+Hess_{g} f-fRic_{g}=g,
\end{equation}
where $Ric$ and $Hess f$ stands, respectively, for the Ricci tensor and Hessian of $f.$

Following the terminology used in \cite{miaotam,miaotamTAMS} we recall the definition of a Miao-Tam critical metrics.

\begin{definition}\label{eq:miaotam}
A Miao-Tam critical metric is a 3-tuple $(M^n,\,g,\,f),$ where $(M^{n},\,g),$ $n\geq3$, is a compact Riemannian manifold with a smooth boundary $\partial M$ and $f: M^{n}\to \Bbb{R}$ is a smooth function such that $f|_{\partial M}=0$ and satisfies the equation (\ref{eqMiaoTam1}).
\end{definition}

\begin{remark}
A fundamental property of a Miao-Tam critical metrics is that its scalar curvature $R_{g}$ is a constant (see Proposition 2.1 in \cite{CEM} or Theorem 7 in \cite{miaotam}).
\end{remark}

Miao and Tam in \cite{miaotam} showed that the only domains in space forms, on which the standard metrics are critical points, are geodesic balls. Based on this (cf. page 156 in \cite{miaotam}), the following question was posed:\\
{\it``It is natural to ask whether they are the only critical points with that boundary condition."}\\
Despite some recent progress, only partial answers have been given to address this issue, see, for instance,  Corollary 3 in \cite{miaotam} and some recent results that we will describe further below. We also remark that Miao and Tam in \cite{miaotamTAMS} constructed nontrivial examples of such critical metrics with connected boundary isometric to a standard sphere which were not geodesic balls in a space form, however the manifolds found were not simply connected. Here, we shall give a partial answer for this problem under harmonicity of Weyl tensor (see Corollary~\ref{thmB}).

In 2011, Miao and Tam  studied under which conditions critical metrics should be warped products. In this case, they were able to construct explicit examples of critical metrics which are in the form of warped products and then classified such critical metrics under the assumption of being locally conformally flat (for more details, see Theorem 1.2 in \cite{miaotamTAMS}). Furthermore, as a consequence of this classification, they proved that a locally conformally flat, simply connected, compact Miao-Tam critical metric with boundary isometric to a standard sphere $\mathbb{S}^{n-1}$ must be isometric to a geodesic ball in a simply connected space form $\mathbb{R}^{n},$ $\mathbb{H}^{n}$ or $\mathbb{S}^{n}.$ In the same article, Miao and Tam studied these critical metrics under Einstein condition, where they were able to remove the condition of boundary isometric to a standard sphere. More recently, the first author and Ribeiro Jr. improved this last result for the Ricci parallel case (see \cite{baltazar} for more details).

Barros, Di\'ogenes and Ribeiro in \cite{BDR}, based on the work of Cao and Chen in \cite{CaoChen}, classified these critical metrics in four dimensional under Bach-flat condition. Nevertheless the authors considered just the case when the Riemannian manifold is simply connected. More precisely, for a $4$-dimensional simply connected manifold with boundary isometric to a standard sphere $\mathbb{S}^{3}$, they replaced the assumption of locally conformally flat in the Miao-Tam result (cf. \cite[Corollary 4.1]{miaotamTAMS}) by the Bach-flat condition. Later on, Kim and Shin \cite{Kim} also in $4$-dimensional case, proved that a  simply connected, compact Miao-Tam critical metric with boundary isometric to a standard sphere $\mathbb{S}^3$ must be isometric to a geodesic ball in a simply connected space form  $\mathbb{R}^{4}$, $\mathbb{H}^{4}$ or $\mathbb{S}^{4}$, provided that the manifold has harmonic curvature. At this point it is important to remember that, if a Riemannian manifold with constant scalar curvature has harmonic curvature, then its Weyl tensor must be harmonic (the converse of this fact is also true for manifold with constant scalar curvature). For more results, please refer to \cite{baltazar,balt17,adam,BDR,BDRR,miaotam,miaotamTAMS}.

Hence, inspired by the above description we shall prove that, in fact, the two geometric conditions are equivalent, that is, a Miao-Tam critical metric has vanishing Bach tensor if and only if its Weyl tensor is harmonic. We emphasize that one of these implications was proved in \cite{BDR}, namely, the authors showed that a Bach-flat Miao-Tam critical metric must have vanishing Cotton tensor, but this is equivalent to harmonicity of Weyl tensor (cf. Lemma 4 combined with Lemma 5 in the referred paper). In order to establish our main result, we will first investigate the Bach-flat Miao-Tam critical metrics on $n$-dimensional manifolds with boundary. In this case, based on the techniques developed in the work of Miao and Tam \cite{miaotamTAMS} and motivated by the work of Qing and Yuan in \cite{QY}, we shall deduce the complete classification of such critical metrics under Bach-flat assumption, see Theorem~\ref{thmbach} in Section 3.

Let us highlight that, the Bochner type formulas has been applied to various problems in global Riemannian geometry. One important application was given recently by the first author and Ribeiro Jr. in \cite{balt17} where, in particular, they obtained some classification results for 3-dimensional Miao-Tam critical metrics. More precisely, in 3-dimensional case, a Miao-Tam critical metric with nonnegative sectional curvature must be isometric to a geodesic ball in a simply connected space form  $\mathbb{R}^{3}$ or $\mathbb{S}^{3}$ (see \cite[Theorem 3]{balt17} for more details).

A well-known fact is that if a $n$-dimensional manifold with $n\geq4,$ is either Einstein, locally conformally flat or has parallel Ricci tensor, then it has harmonic Weyl tensor. Therefore, it is natural to ask what happens to the geometry of the Miao-Tam critical metrics with harmonic Weyl tensor. To do so, as application of the Bochner type formula obtained in \cite{balt17}, we shall provide some important lemmas in order to prove that such critical metrics must satisfy $W(\cdot,\nabla f,\cdot,\nabla f)=0$ on $M,$ where $f$ is its potential function. Finally, we will be able to classify such critical metrics under harmonicity of Weyl tensor. More precisely, we will prove the following result.

\begin{theorem}\label{thmdivW}
Let $(M^n,\,g,\,f)$ be a Miao-Tam critical metric with harmonic Weyl tensor. Suppose that the first Dirichlet eigenvalue of $(n-1)\Delta_{g}+R$ is nonnegative, where $R$ is the scalar curvature of $g$.
\begin{enumerate}
  \item[(i)] If $\Sigma$ is disconnected, then $\Sigma$ has exactly two connected components and $(M,g)$ is isometric to $(I\times N,ds^{2}+r^{2}h)$ where $I$ is a finite interval in $\mathbb{R}$ containing the origin $0$, $(N,h)$ is a compact (without boundary) Einstein manifold with $Ric=(n-2)\kappa_{0}h,$ for some constant $\kappa_{0}$, $r$ is a positive function on $I$ satisfying $r'(0)=0$ and
$$r''+\frac{R}{n(n-1)}r=ar^{1-n}$$
for some constant $a>0$, and the constant $\kappa_{0}$ satisfies
$$(r')^{2}+\frac{R}{n(n-1)}r^{2}+\frac{2a}{n-2}r^{2-n}=\kappa_{0}.$$
  \item[(ii)] If $\Sigma$ is connected, then $(M,g)$ is either isometric to a geodesic ball in a simply connected space form $\mathbb{R}^{n}$, $\mathbb{H}^{n}$, $\mathbb{S}^{n}$, or $(M,g)$ is covered by one of the above mentioned warped product in (i) with a covering group $\mathbb{Z}_{2}.$
\end{enumerate}
\end{theorem}

As an immediate consequence of Theorem \ref{thmdivW} we deduce the following rigidity result.

\begin{corollary}\label{thmB}
Let $(M^n,\,g,\,f)$ be a simply connected, Miao-Tam critical metric with harmonic Weyl tensor and boundary isometric to a standard sphere $\mathbb{S}^{n-1}.$  Then $(M^n,\,g)$ is isometric to a geodesic ball in a simply connected space form $\Bbb{R}^{n},$ $\Bbb{H}^{n},$ or $\Bbb{S}^{n}.$
\end{corollary}

\begin{remark}
Recently, the first author, R. Di\'ogenes and Ribeiro Jr. \cite{BaltRR} motivated by \cite{Kim} proved that a simply connected Miao-Tam critical metric satisfying the second order divergence-free Weyl tensor condition and boundary isometric to a standard sphere $\mathbb{S}^{3}$ must be isometric to a geodesic ball in a simply connected space form $\Bbb{R}^{4},$ $\Bbb{H}^{4},$ or $\Bbb{S}^{4}.$ Therefore, taking into account \cite[Lemma 4]{BaltRR} and Corollary~\ref{thmB} above, it is not difficult to check that their result must be true in all dimensions.

\end{remark}

\section{Background}
\label{Preliminaries}

In this section we will recall some information and basic results that will be useful in the proof of our main theorem. We remember that the fundamental equation of Miao-Tam critical metric, i.e.,
\begin{equation}\label{eqfund1}
-(\Delta f)g+Hess f-fRic=g.
\end{equation}
Taking the trace of (\ref{eqfund1}) we arrive at
\begin{equation}
\label{eqtrace} \Delta f +\frac{fR+n}{n-1}=0.
\end{equation}
Moreover, we recall that for operators $S,T:\mathcal{H} \to \mathcal{H}$ defined over an $n$-dimensional Hilbert space $\mathcal{H}$ the Hilbert-Schmidt inner product is defined according to
\begin{equation}
\langle S,T \rangle =\rm tr\big(ST^{*}\big), \label{inner}
\end{equation}
where $\rm tr$ and $*$ denote, respectively, the trace and the adjoint operation. Moreover, if $I$ denotes the identity operator on $\mathcal{H}$ the traceless operator of $T$ is given by
 \begin{equation}
 \label{eqtr1}
\mathring{T}=T - \frac{\rm tr T}{n}I.
 \end{equation}

By using (\ref{eqtrace}) it is not difficult to check that
\begin{equation}
\label{IdRicHess} f\mathring{Ric}=\mathring{Hess f}.
\end{equation}

In order to proceed, we recall some special tensors as well as some terminology in the study of curvature for a Riemannian manifold $(M^n,\,g),\,n\ge 3.$ The Weyl tensor $W$ is defined by the following decomposition formula
\begin{eqnarray}
\label{weyl}
R_{ijkl}&=&W_{ijkl}+\frac{1}{n-2}\big(R_{ik}g_{jl}+R_{jl}g_{ik}-R_{il}g_{jk}-R_{jk}g_{il}\big) \nonumber\\
 &&-\frac{R}{(n-1)(n-2)}\big(g_{jl}g_{ik}-g_{il}g_{jk}\big),
\end{eqnarray}
where $R_{ijkl}$ stands for the Riemannian curvature operator. It is important to detach that the Weyl tensor has the same symmetries properties of the Riemann tensor and it has trace-free in any two indices.

Moreover, the Cotton tensor $C$ is defined as follows
\begin{equation}
\label{cotton}
C_{ijk}=\nabla_{i}R_{jk}-\nabla_{j}R_{ik}-\frac{1}{2(n-1)}\big(\nabla_{i}R g_{jk}-\nabla_{j}R
g_{ik}).
\end{equation}
When $n\geq 4$ we have
\begin{equation}
\label{cottonwyel} C_{ijk}=-\frac{(n-2)}{(n-3)}\nabla_{l}W_{ijkl}.
\end{equation}

An important remark about Cotton tensor $C_{ijk}$ is that it is skew-symmetric in the first two indices and
trace-free in any indices, that is,
\begin{equation}
\label{cottonprop} C_{ijk}=-C_{jik} \,\,\,\,\,\, {\rm and} \,\,\,\,\,\, g^{ij}C_{ijk}=g^{ik}C_{ijk}=0.
\end{equation}
For more details about these tensors we address to \cite{besse}.

Finally, we recall the well-known Bach tensor that was introduced by Bach in \cite{bach}. On a Riemannian manifold $(M^n,g)$, $n\geq 4,$ the Bach tensor is defined in term of the components of the Weyl tensor $W_{ikjl}$ as follows
\begin{equation}
\label{bach} B_{ij}=\frac{1}{n-3}\nabla_{k}\nabla_{l}W_{ikjl}+\frac{1}{n-2}R_{kl}W_{ikjl},
\end{equation}
while for $n=3$ it is given by
\begin{equation}
\label{bach3} B_{ij}=\nabla_{k}C_{kij}.
\end{equation} We say that $(M^n,g)$ is Bach-flat when $B_{ij}=0.$ It is easy to check that locally conformally flat metrics as well as Einstein metrics are Bach-flat. It is worth to point out that in dimension $4$, we have that, on any compact manifold $(M^{4},g)$, Bach-flat metrics are precisely the critical points of the {\it conformally invariant} functional on the space of the metrics,
$$\mathcal{W}(g)=\int_{M}|W_{g}|^{2}dV_{g}.$$
For more details see, for example \cite{besse} or \cite{derd1}.

Next, from commutation formulas for first covariant derivative of the Ricci curvature, for any Riemannian manifold $M^n,$ we have
\begin{equation}\label{idRicci}
\nabla_{i}\nabla_{j}R_{pq}-\nabla_{j}\nabla_{i}R_{pq}=R_{ijps}R_{sq}+R_{ijqs}R_{ps}.
\end{equation}
For more details see \cite{chow,Via}.

With this notation in mind, we may deduce the following formula for the laplacian of the norm of the Ricci tensor. Since the proof is short, for the sake of completeness, we include it.

\begin{lemma}\label{deltaRICCI}
Let $(M^{n},g)$ be a Riemannian manifold. Then we have:
\begin{eqnarray*}
\Delta |Ric|^{2}&=&2|\nabla Ric|^{2}+2\nabla_{p}(C_{pij}R_{ij})-|C_{ijk}|^{2}+2(R_{ij}R_{ik}R_{jk}-R_{ik}R_{jl}R_{ijkl})\\
&&-\frac{n}{2(n-1)}|\nabla R|^{2}+\frac{1}{n-1}div((n-2)Ric(\nabla R)+R\nabla R).
\end{eqnarray*}
In particular, if $(M,g)$ has a constant scalar curvature, the expression above becomes
\begin{eqnarray*}
\Delta |Ric|^{2}&=&2|\nabla Ric|^{2}+2\nabla_{p}(C_{pij}R_{ij})-|C_{ijk}|^{2}+2(R_{ij}R_{ik}R_{jk}-R_{ik}R_{jl}R_{ijkl}).
\end{eqnarray*}
\end{lemma}

\begin{proof}
To begin with, we use (\ref{cotton}) to arrive at
\begin{eqnarray*}
\frac{1}{2}\Delta|Ric|^{2}&=&|\nabla Ric|^{2}+R_{ij}\nabla_{p}\nabla_{p}R_{ij}\\
&=&|\nabla Ric|^{2}+R_{ij}\nabla_{p}(C_{pij}+\nabla_{i}R_{pj}+\frac{1}{2(n-1)}(\nabla_{p}Rg_{ij}-\nabla_{i}Rg_{pj}))\\
&=&|\nabla Ric|^{2}+R_{ij}\nabla_{p}C_{pij}+R_{ij}\nabla_{p}\nabla_{i}R_{pj}+\frac{1}{2(n-1)}(R\Delta R-\nabla_{i}\nabla_{j}RR_{ij}).
\end{eqnarray*}
Then, from (\ref{idRicci}) we obtain
\begin{eqnarray*}
\frac{1}{2}\Delta|Ric|^{2}&=&|\nabla Ric|^{2}+R_{ij}\nabla_{p}C_{pij}+\nabla_{i}\nabla_{p}R_{pj}R_{ij}+(R_{ij}R_{ik}R_{jk}-R_{ik}R_{jl}R_{ijkl})\\
&&+\frac{1}{2(n-1)}(R\Delta R-\nabla_{i}\nabla_{j}RR_{ij}),
\end{eqnarray*}
where we change some indices for simplicity.

Now, using the twice contracted second Bianchi identity, we immediately have
\begin{eqnarray*}
\frac{1}{2}\Delta|Ric|^{2}&=&|\nabla Ric|^{2}+R_{ij}\nabla_{p}C_{pij}+\frac{n-2}{2(n-1)}\nabla_{i}\nabla_{j}RR_{ij}+\frac{1}{2(n-1)}R\Delta R\\
&&+(R_{ij}R_{ik}R_{jk}-R_{ik}R_{jl}R_{ijkl})\\
&=&|\nabla Ric|^{2}+\nabla_{p}(C_{pij}R_{ij})-C_{pij}\nabla_{p}R_{ij}+(R_{ij}R_{ik}R_{jk}-R_{ik}R_{jl}R_{ijkl})\\
&&+\frac{n-2}{2(n-1)}(div(Ric(\nabla R))-\frac{1}{2}|\nabla R|^{2})+\frac{1}{2(n-1)}(div(R\nabla R)-|\nabla R|^{2})\\
&=&|\nabla Ric|^{2}+\nabla_{p}(C_{pij}R_{ij})-C_{pij}\nabla_{p}R_{ij}+(R_{ij}R_{ik}R_{jk}-R_{ik}R_{jl}R_{ijkl})\\
&&-\frac{n}{4(n-1)}|\nabla R|^{2}+\frac{1}{2(n-1)}div((n-2)Ric(\nabla R)-R\nabla R).
\end{eqnarray*}
To finalize, it suffices to use that the Cotton tensor is skew-symmetric in the first two indices and has trace-free in any two indices.

\end{proof}

Now, for a Miao-Tam critical metric, we recall the following 3-tensor defined in \cite{BDR},
\begin{eqnarray}\label{TensorT}
T_{ijk}&=&\frac{n-1}{n-2}(R_{ik}\nabla_{j}f-R_{jk}\nabla_{i}f)-\frac{R}{n-2}(g_{ik}\nabla_{j}f-g_{jk}\nabla_{i}f)\nonumber\\
&&+\frac{1}{n-2}(g_{ik}R_{js}\nabla_{s}f-g_{jk}R_{is}\nabla_{s}f).
\end{eqnarray}
Note that $T_{ijk}$ has the same symmetry properties as the Cotton tensor:
$$T_{ijk}=-T_{jik}\;\;\;and\;\;\;g^{ij}T_{ijk}=g^{ik}T_{ijk}=0.$$
Furthermore, with a straightforward computation we verify that
\begin{equation}\label{CTW}
fC_{ijk}=T_{ijk}+W_{ijkl}\nabla_{l}f.
\end{equation}

To conclude this section, the following lemma will be useful for our purpose.
\begin{lemma}\label{bachCotton}
Let $(M^{n},g,f)$ be a Miao-Tam critical metric. Then the following identity holds:
$$(n-2)fB_{ij}=\nabla_{k}T_{kij}+\frac{n-3}{n-2}C_{jki}\nabla_{k}f+C_{ikj}\nabla_{k}f.$$
\end{lemma}
\begin{proof}
Note that, by using (\ref{cottonwyel}), $B_{ij}$ can be rewritten as
\begin{equation}\label{bacha}
(n-2)B_{ij}=\nabla_{k}C_{kij}+W_{ikjl}R_{kl}.
\end{equation}
Now, from (\ref{eqfund1}) we obtain
\begin{eqnarray*}
(n-2)fB_{ij}&=&f\nabla_{k}C_{kij}+W_{ikjl}\nabla_{k}\nabla_{l}f\\
&=&\nabla_{k}(fC_{kij})-C_{kij}\nabla_{k}f+W_{ikjl}\nabla_{k}\nabla_{l}f.
\end{eqnarray*}
Therefore, we use (\ref{CTW}) combined with (\ref{cottonwyel}) to infer
\begin{eqnarray*}
(n-2)fB_{ij}&=&\nabla_{k}T_{kij}+\nabla_{k}(W_{kijl}\nabla_{l}f)-C_{kij}\nabla_{k}f+W_{ikjl}\nabla_{k}\nabla_{l}f\\
&=&\nabla_{k}T_{kij}+\frac{n-3}{n-2}C_{jki}\nabla_{k}f-C_{kij}\nabla_{k}f,
\end{eqnarray*}
as desired.
\end{proof}

\section{Bach-flat Critical Metrics}

As mentioned in the introduction, this article has been motivated by the works of Miao and Tam \cite{miaotamTAMS} as well as  Barros, Di\'ogenes and Ribeiro \cite{BDR}. However, we must emphasize that the works of Kobayashi \cite{kobayashi} and Kobayashi and Obata \cite{obata} on classification of locally conformally flat static metrics and more recently,  Qing and Yuan \cite{QY} for Bach-flat static metrics, had great influence in the accomplishment of this article. This occurs because some expressions obtained are similar to those found in the case of Miao-Tam critical metrics. For example, if we consider a Miao-Tam critical metric as a warped product then we must have that the warped function satisfies the same ODE obtained in the static case (compare Lemma 1.1 in \cite{kobayashi} with proposition 3.1 in \cite{miaotamTAMS}).

Throughout this section we will consider $(M^{n},g,f)$, ($n\geq3$), a $n$-dimensional Miao-Tam critical metric with smooth boundary $\partial M$ and constant scalar curvature $R$ such that the first Dirichlet eigenvalue of $(n-1)\Delta_{g}+R$ on $M$ is positive. Thus, by \cite{miaotam} the potential function $f$ satisfies $f>0$ in the interior of $M$ and, if $\nu$ denotes the outward unit normal to $\partial M,$ then $\langle\nabla f,\nu\rangle<0$ on each connected component of $\partial M.$

Let us recall some observations for a Bach-flat Miao-Tam critical metric. In fact, in \cite{BDR} the authors proved that Bach-flat condition implies that the auxiliary tensor $T_{ijk}$ and the Cotton tensor $C_{ijk}$  vanish completely. So, it is immediate to check that, in the Bach-flat case, the Weyl tensor satisfies
\begin{eqnarray}\label{Weylf}
W_{ijkl}\nabla_{l}f=0,
\end{eqnarray}
for all $0\leq i,j,k\leq n.$

Proceeding we recall that, at a regular point of the potential function $f,$ the vector field $\nu=-\frac{\nabla f}{|\nabla f|}$ is normal to level set $\Sigma_{c}=\{p\in M:f(p)=c\}.$ So, considering $\{e_{1}=\nu,e_{2},\dots,e_{n}\}$ as an orthonormal frame with $\{e_{a}\}_{a\geq2}$ tangent to $\Sigma$, the second fundamental form of $\Sigma_{c}$ is given by

\begin{equation}\label{SFF}
\mathbb{II}_{ab}=\langle \nabla_{e_{a}}e_{1},e_{b}\rangle=-\frac{1}{|\nabla f|}\nabla_{a}\nabla_{b}f,
\end{equation}
and its mean curvature $H$ is given as follows,

\begin{equation}\label{meancurvature}
H=\frac{1}{|\nabla f|}(\nabla_{1}\nabla_{1}f-\Delta f)=\frac{1}{|\nabla f|}(fR_{11}+1).
\end{equation}
Thus, its not difficult to check that
\begin{eqnarray*}
|\nabla f|^{2}\sum_{a,b=2}^{n}|\mathbb{II}_{ab}-\frac{H}{n-1}g_{ab}|^{2}&=&f^{2}|Ric|^{2}-\frac{nR_{11}^{2}f^{2}+f^{2}R^{2}-2f^{2}RR_{11}}{n-1}\nonumber\\
&&-2f^{2}\sum_{a=2}^{n}R_{1a}^{2}
\end{eqnarray*}
and, defining the function $\rho=|\nabla f|^{2}+\frac{2}{n-1}f+\frac{R}{n-1}f^{2},$ after some computation, we rewrite the above expression as
\begin{eqnarray}\label{normsegforma}
|\nabla f|^{4}\sum_{a,b=2}^{n}|\mathbb{II}_{ab}-\frac{H}{n-1}g_{ab}|^{2}&=&f^{2}|\nabla f|^{2}|Ric|^{2}-\frac{1}{2}|\nabla^{\Sigma}\rho|^{2}-\frac{R^{2}|\nabla f|^{2}f^{2}}{n-1}\nonumber\\
&&-\frac{n}{4(n-1)|\nabla f|^{2}}\langle\nabla\rho,\nabla f\rangle^{2}+\frac{Rf}{n-1}\langle\nabla\rho,\nabla f\rangle.
\end{eqnarray}

Hence, if we compare Eq. (\ref{normsegforma}) with the norm of the tensor $T_{ijk},$ we can provide a key property for our purpose. Such relation was obtained by Barros, Di\'ogenes and Ribeiro jr (see \cite{BDR} for more details).

\begin{lemma}\label{normaTaux}
Let $(M^{n},g,f)$ be a Miao-Tam critical metric. Let $\Sigma_{c}=\{p\in M, f(p)=c\}$ be a level set of $f.$ If $g_{ab}$ denotes the induced metric on $\Sigma_{c},$ then, at any point where $\nabla f\neq0,$ we have
\begin{equation}
|fT|^{2}=\frac{2(n-1)^{2}}{(n-2)^{2}}|\nabla f|^{4}\sum_{a,b=2}^{n}|\mathbb{II}_{ab}-\frac{H}{n-1}g_{ab}|^{2}+\frac{n-1}{2(n-2)}|\nabla^{\Sigma}\rho|^{2},
\end{equation}
where $\mathbb{II}_{ab}$ and $H$ are the second fundamental form and the mean curvature of $\Sigma_{c},$ respectively.
\end{lemma}

Furthermore, still in \cite{BDR}, the authors showed nice properties about the geometry of a Miao-Tam critical metric $(M^{n},g,f)$ and level sets of the potential function $f,$ namely, they obtained the following lemma.
\begin{lemma}\label{levelsetgeral}
	Let $(M^n, g, \nabla f)$  be a Miao–Tam critical metric with  $T_{ijk}=0$. Let $c$ be a regular value of $f$ and $\Sigma_{c}=\{p\in M, f(p)=c\}$ be a level set of $f$. We consider $e_{1}=-\frac{\nabla f}{|\nabla f|}$ and choose an orthonormal frame $\{e_{2}, ..., e_{n}\}$ tangent to $\Sigma_{c}$. Under these conditions the following assertions hold
	\begin{enumerate}
        \item The second fundamental form $\mathbb{II}_{ab}$ of $\Sigma_{c}$ is $\mathbb{II}_{ab}=\dfrac{H}{n-1}g_{ab}$;
    	\item $|\nabla f|$ is constant on $\Sigma_{c}$;
        \item $R_{1a}=0,$ for all $a\geq2$ and $e_{1}$ is an eigenvector of $Ric$;
        \item The mean curvature of $\Sigma_{c}$ is constant;
        \item On $\Sigma_{c}$, the Ricci tensor either has a unique eigenvalue or two distinct eigenvalues with multiplicity $1$ and $n-1$. Moreover, the eigenvalue with multiplicity is in the direction of $\nabla f;$
        \item $R_{1abc}=0$ for $a,b,c\in \{2,\ldots,n\}$.
	\end{enumerate}
\end{lemma}
In the sequel, with the above notation, we are able to prove that the level sets are Einstein manifolds since we assume the Bach-flat condition. More precisely, we have the following result.

\begin{lemma}\label{lemkey}
Let $(M,g,f)$ be a Bach-flat Miao-Tam critical metric. Then for any level set $\Sigma_{c}$ with $c\neq0$ we have
$$R^{\Sigma}_{ab}=\frac{1}{n-1}\Big[R-2R_{11}+\frac{(n-2)H^{2}}{n-1}\Big]g_{ab},$$
where $R_{11}$ and $H$ are constant on $\Sigma_{c}.$ In particular, the level sets $\Sigma_{c}$ ($c\neq0$) with induced metric are Einstein manifolds.
\end{lemma}
\begin{proof}
Firstly, we use the Gauss equation jointly with Lemma~\ref{levelsetgeral}, item $(1),$ to deduce
\begin{eqnarray}\label{RSigma}
R^{\Sigma}_{abcd}&=&R_{abcd}+\mathbb{II}_{ac}\mathbb{II}_{bd}-\mathbb{II}_{ad}\mathbb{II}_{bc}\nonumber\\
&=&R_{abcd}+\frac{H^{2}}{(n-1)^{2}}(g_{ac}g_{bd}-g_{ad}g_{bc}).
\end{eqnarray}
Now, substituting (\ref{weyl}) into (\ref{RSigma}), we get
\begin{eqnarray*}
R^{\Sigma}_{abcd}&=&W_{abcd}+\frac{1}{n-2}(R_{ac}g_{bd}+R_{bd}g_{ac}-R_{ad}g_{bc}-R_{bc}g_{ad})\\
&&-\frac{R}{(n-1)(n-2)}(g_{ac}g_{bd}-g_{ad}g_{bc})+\frac{H^{2}}{(n-1)^{2}}(g_{ac}g_{bd}-g_{ad}g_{bc}).
\end{eqnarray*}
Next, substituting (\ref{eqfund1}) into (\ref{SFF}) we obtain
\begin{equation}\label{SFFaux}
\mathbb{II}_{ab}=\frac{1}{|\nabla f|}\Big[-fR_{ab}+\frac{Rf+1}{n-1}g_{ab}\Big].
\end{equation}
On the other hand, by Lemma~\ref{levelsetgeral}, item $(1),$ jointly with the expression of the mean curvature obtained in (\ref{meancurvature}), we can deduce another expression for the second fundamental form, i.e.,

\begin{equation}\label{SFFauxx}
\mathbb{II}_{ab}=\frac{1+fR_{11}}{(n-1)|\nabla f|}g_{ab}.
\end{equation}

So, comparing (\ref{SFFaux}) with (\ref{SFFauxx}) and using that $f>0$ in the interior of $M$, we arrive at
$$R_{ab}=\frac{R-R_{11}}{n-1}g_{ab},$$
for all $2\leq a,b\leq n.$ Consequently, we have that

\begin{eqnarray}\label{Riemannaux1}
R^{\Sigma}_{abcd}&=&W_{abcd}+\frac{2(R-R_{11})}{(n-1)(n-2)}(g_{ac}g_{bd}-g_{ad}g_{bc})\nonumber\\
&&-\frac{R}{(n-1)(n-2)}(g_{ac}g_{bd}-g_{ad}g_{bc})+\frac{H^{2}}{(n-1)^{2}}(g_{ac}g_{bd}-g_{ad}g_{bc})\nonumber\\
&=&W_{abcd}+\frac{1}{(n-1)(n-2)}\Big[R-2R_{11}+\frac{(n-2)H^{2}}{n-1}\Big](g_{ac}g_{bd}-g_{ad}g_{bc}).
\end{eqnarray}
In particular, as we already know that $W_{ijkl}\nabla_{l}f=0$ (see equality (\ref{Weylf})), we immediately obtain
\begin{eqnarray}\label{ricsigma}
R^{\Sigma}_{ac}=\frac{1}{n-1}\Big[R-2R_{11}+\frac{(n-2)H^{2}}{n-1}\Big]g_{ac}
\end{eqnarray}
and
\begin{eqnarray}\label{Ssigma}
R^{\Sigma}=R-2R_{11}+\frac{(n-2)H^{2}}{n-1}.
\end{eqnarray}

To finalize, in order to see that $R_{11}$ and the mean curvature $H$ are constants on the level sets, it suffices to use Lemma~\ref{levelsetgeral} jointly with the fact that $M$ has vanishing Cotton tensor. So, we complete the proof of the lemma.
\end{proof}

Now, let us recall a fundamental proposition that determine sufficient conditions to a warped product metric becomes a Miao-Tam critical metric (For detailed calculations, please refer to \cite{miaotamTAMS}, Proposition 3.1).

\begin{proposition}\label{p1}
Let $I\times_{r}N=(I\times N,g=dt^{2}+r(t)^{2}g_{N})$ be a warped product metric. Then, for any constant $R$, the metric $g$ has constant scalar curvature $R$ and satisfies (\ref{eqfund1}) for a smooth function $f$ depending only on $t\in I$ if and only if the following holds:
\begin{enumerate}
  \item[(i)] $(N,g_{N})$ is an Einstein manifold with $Ric_{g_{N}}=(n-2)\kappa_{0}g_{N},$ the function $r$ satisfies
  $$r''+\frac{R}{n(n-1)}r=ar^{1-n}$$
  for some constant $a$, and the constant $\kappa_{0}$ satisfies
  $$(r')^{2}+\frac{R}{n(n-1)}r^{2}+\frac{2a}{n-2}r^{2-n}=\kappa_{0}.$$
  \item[(ii)] The function $f$ satisfies
  $$r'f'-r''f=-\frac{1}{n-1}r.$$
\end{enumerate}

\end{proposition}

The next result plays an important role in our main theorem of this section. In what follows, consider $\Sigma_{0}$ as a connected component of the boundary $f^{-1}(0)$ and $\widetilde{M}_{0}$ be the connected component of the open set $\{x\in M;|\nabla f|>0\}$ such that its closure contains $\Sigma_{0}$ and let $M_{0}=\widetilde{M}_{0}\cup\Sigma_{0}.$ With this consideration in mind, the same spliting result, obtained by Miao and Tam in \cite[Lemma 4.2]{miaotamTAMS}, holds for a critical metric under the assumption which the manifold is Bach-flat.

\begin{proposition}\label{p2}
Let $(M^{n},g,f)$ be a Bach-flat Miao-Tam critical metric.  Then, there exist a constant $\delta_{0}>0$ such that $(M_{0},g)$ is isometric to a warped product $([0,\delta_{0})\times\Sigma_{0},\; dt^{2}+r^{2}g_{\Sigma_{0}}),$ where $r>0$ is a smooth warped function on $[0,\delta_{0})$ and $g_{\Sigma_{0}}$ is the induced metric on $\Sigma_{0}$ from $g.$ Moreover, $f$ on $M_{0}$ depends only on $t\in[0,\delta)$ and $\Sigma_{0}$ is an Einstein manifold.
\end{proposition}

\begin{remark}
Let us highlight that the proposition~\ref{p1} and proposition~\ref{p2} jointly with Lemma 4.4 in \cite{miaotamTAMS} give a complete solution to the fundamental equation (\ref{eqfund1}) for a Bach-flat case except in the critical set of $f.$
\end{remark}

\begin{remark}
We point out that our warped product critical metrics, obtained under Bach-flat assumption, has level sets as an Einstein manifold instead of constant sectional curvature as in \cite{miaotamTAMS}.
\end{remark}

Therefore, after this local splitting result, we are ready to build the complete solution to a Miao-Tam critical metric under Bach-flat assumption, proceeding exactly as in \cite{miaotamTAMS}. Indeed, we use  Lemma 4.4 in \cite{miaotamTAMS} in order to get that the boundary has at most two connected components and, then, it suffices to use the Theorems 4.1 and 4.2 in the referred paper to conclude this strong result. More precisely, we have established the following result.

\begin{theorem}\label{thmbach}
Let $(M^n,\,g,\,f)$ be a Bach-flat Miao-Tam critical metric. Suppose that the first Dirichlet eigenvalue of $(n-1)\Delta_{g}+R$ is nonnegative, where $R$ is the scalar curvature of $g$.
\begin{enumerate}
  \item[(i)] If $\Sigma$ is disconnected, then $\Sigma$ has exactly two connected components and $(M,g)$ is isometric to $(I\times N,ds^{2}+r^{2}h)$ where $I$ is a finite interval in $\mathbb{R}$ containing the origin $0$, $(N,h)$ is a compact (without boundary) Einstein manifold with $Ric=(n-2)\kappa_{0}h,$ for some constant $\kappa_{0},$ $r$ is a positive function on $I$ satisfying $r'(0)=0$ and
$$r''+\frac{R}{n(n-1)}r=ar^{1-n}$$
for some constant $a>0$, and the constant $\kappa_{0}$ satisfies
$$(r')^{2}+\frac{R}{n(n-1)}r^{2}+\frac{2a}{n-2}r^{2-n}=\kappa_{0}.$$
  \item[(ii)] If $\Sigma$ is connected, then $(M,g)$ is either isometric to a geodesic ball in a simply connected space form $\mathbb{R}^{n}$, $\mathbb{H}^{n}$, $\mathbb{S}^{n}$, or $(M,g)$ is covered by one of the above mentioned warped product in (i) with a covering group $\mathbb{Z}_{2}.$
\end{enumerate}
\end{theorem}

As an immediate consequence of Theorem \ref{thmbach} we deduce the following rigidity result.

\begin{corollary}\label{corB}
Let $(M^n,\,g,\,f)$ be a simply connected, Bach-flat Miao-Tam critical metric with boundary isometric to a standard sphere $\mathbb{S}^{n-1}.$  Then $(M^n,\,g)$ is isometric to a geodesic ball in a simply connected space form $\Bbb{R}^{n},$ $\Bbb{H}^{n},$ or $\Bbb{S}^n.$
\end{corollary}

\section{Critical metrics with harmonic Weyl tensor}

In this section, we focus on the classification of Miao-Tam critical metric with harmonic Weyl tensor. The main idea is to use the methods of Yun, Chang and Hwang \cite{s1,s2}, where the authors classified the well-known critical point equation (CPE metrics) under harmonicity of the Weyl tensor.

In what follows we shall consider $(M^{n},g)$, ($n\geq3$), as a $n$-dimensional Miao-Tam critical metric with harmonic Weyl tensor, (i.e., $divW=0$), smooth boundary $\partial M$ and constant scalar curvature $R$ such that the first Dirichlet eigenvalue of $(n-1)\Delta_{g}+R$ on $M$ is positive. As mentioned before, the potential function $f$ satisfies $f>0$ in the interior of $M$ and, if $\nu$ denotes the outward unit normal to $\partial M,$ then $\langle\nabla f,\nu\rangle<0$ on each connected component of $\partial M.$

In order to prove the Theorem~\ref{thmdivW}, we will initially show that at a regular point $q\in M$, we obtain $\nabla f$ as an eigenvector for Ricci tensor.

\begin{lemma}\label{autoRIC}
Let $(M^{n},g,f)$ be a Miao-Tam critical metric with $divW=0.$ Suppose $q\in M $ is a regular point of $f$, i.e., $|\nabla f|(q)\neq0.$ Then, $\nabla f$ is an eigenvector for Ric.
\end{lemma}
\begin{proof}
In fact,  we consider  an orthonormal frame $\{e_{1},\ldots,e_{n}\}$ diagonalizing $Ric$ at a point $q\in M^{n}$ such that $\nabla f(q)\neq0$, namely, we have that $Ric(e_{i})=\alpha_{i}e_{i}$, where $\alpha_{i}$ are the associated eigenvalues.

Now, from (\ref{cottonwyel}) jointly with Eq. (\ref{CTW}) and our hypotheses about Weyl tensor, we deduce that
$$0=fC_{ijk}\nabla_{k}f=T_{ijk}\nabla_{k}f$$
and, consequently, we use the definition (\ref{TensorT}) to infer
$$R_{jk}\nabla_{k}f\nabla_{i}f-R_{ik}\nabla_{k}f\nabla_{j}f=0.$$
Computing this last expression at $q\in M,$ we get
\begin{equation}\label{lambda}
(\alpha_{j}-\alpha_{i})\nabla_{i}f\nabla_{j}f=0.
\end{equation}
Hence, if we consider the following nonempty set $L=\{i;\nabla_{i}f\neq0\}$ then, from (\ref{lambda}) we have $\alpha_{i}=\alpha$, for all $i\in L,$ and consequently
$$Ric(\nabla f)=Ric\Big(\displaystyle\sum_{i\in L}\nabla_{i}fe_{i}\Big)=\sum_{i\in L}\nabla_{i}f\alpha_{i}e_{i}=\alpha\nabla f.$$
This is what we want to prove.
\end{proof}

Now, we shall obtain some properties at a regular point of $M^{n}$ under harmonicity of Weyl tensor. In what follows, let $q\in M^{n}$ be an arbitrary regular point and
consider an orthonormal frame  $\{e_{1},e_{2}, ..., e_{n}\}$ in a neighborhood of $q$ contained in $M$ such that $e_{1}=-\frac{\nabla f}{|\nabla f|}.$

Firstly, by direct computation using the twice contracted second Bianchi identity together with the Equations (\ref{eqfund1}) and (\ref{IdRicHess}), we get
\begin{eqnarray}\label{auxdivRIC1}
{\rm div}(\mathring{Ric}\nabla f)=\nabla_{i}(\mathring{R}_{ij}\nabla_{j}f)=\mathring{R}_{ij}\nabla_{i}\nabla_{j}f=f|\mathring{Ric}|^{2}.
\end{eqnarray}
On the other hand, from Lemma~\ref{autoRIC} we immediately have $\mathring{Ric}(\nabla f)=\mathring{R}_{11}\nabla f$ and, consequently, we may deduce
\begin{eqnarray}\label{auxdivRIC2}
{\rm div}(\mathring{Ric}\nabla f)&=&\nabla_{i}(\mathring{R}_{11}\nabla_{i} f)\nonumber\\
&=&\langle\nabla\mathring{R}_{11},\nabla f\rangle+\mathring{R}_{11}\Delta f\nonumber\\
&=&\langle\nabla\mathring{R}_{11},\nabla f\rangle-\frac{Rf+n}{n-1}\mathring{R}_{11}.
\end{eqnarray}
Hence, comparing (\ref{auxdivRIC1}) with (\ref{auxdivRIC2}), we get the following identity
\begin{equation}\label{normafRIC}
f|\mathring{Ric}|^{2}=\langle\nabla\mathring{R}_{11},\nabla f \rangle -\frac{Rf+n}{n-1}\mathring{R}_{11}.
\end{equation}

In the sequel, from (\ref{CTW}) jointly with the fact that $\nabla f$ is an eigenvector for Ric, we obtain  a very useful equality for our purpose, namely, for all $2\leq i,j\leq n,$
\begin{eqnarray*}
0&=&T_{i1j}+W_{i1j1}|\nabla f|\\
&=&\frac{n-1}{n-2}\left(R_{ij}-\frac{R}{n-1}g_{ij}+R_{11}g_{ij}+\frac{n-2}{n-1}W_{i1j1}\right)|\nabla f|\\
&=&\frac{n-1}{n-2}\left(\mathring{R}_{ij}+\frac{\mathring{R}_{11}}{n-1}g_{ij}+\frac{n-2}{n-1}W_{i1j1}\right)|\nabla f|,
\end{eqnarray*}
that is,
\begin{equation}\label{auxRICCI}
\mathring{R}_{ij}+\frac{\mathring{R}_{11}}{n-1}g_{ij}+\frac{n-2}{n-1}W^{\nu}_{ij}=0,
\end{equation}
for all $2\leq i,j\leq n$, where $W^{\nu}_{ij}=W_{i1j1}.$
In particular, we get the following expression for the squared norm of the traceless Ricci tensor:
\begin{eqnarray}\label{auxnormaRIC}
|\mathring{Ric}|^{2}&=&\mathring{R}_{11}^{2}+\sum_{i,j\geq2}^{n}\mathring{R}_{ij}^{2}\nonumber\\
&=&\mathring{R}_{11}^{2}+\sum_{i,j\geq2}^{n}\left(\frac{\mathring{R}_{11}}{n-1}g_{ij}\right)^{2}+\left(\frac{n-2}{n-1}\right)^{2}\sum_{i,j\geq2}^{n}W_{i1j1}^{2}\nonumber\\
&=&\frac{n}{n-1}\mathring{R}_{11}^{2}+\left(\frac{n-2}{n-1}\right)^{2}|W^{\nu}|^{2}.
\end{eqnarray}

We now gives a couple of lemmas which are essential to conclusion of the main result of this paper.

\begin{lemma}\label{nablafRic}
Let $(M,g,f)$ be a Miao-Tam critical metric with harmonic Weyl tensor. Then, for any regular point, we have
$$\frac{1}{2}\nabla f(|\mathring{Ric}|^{2})=2f\mathring{R}_{11}|\mathring{Ric}|^{2}-ftr(\mathring{Ric}^{3})-\Delta f \mathring{R}_{11}^{2}-\frac{\Delta f}{n}|\mathring{Ric}|^{2}.$$
\end{lemma}
\begin{proof}
To begin with, since the Cotton tensor vanishes on $M^{n}$ (see Eq. (\ref{cottonwyel})), we deduce that
\begin{eqnarray*}
\frac{1}{2}\nabla f(|\mathring{Ric}|^{2})&=&\mathring{R}_{ij}\nabla_{k}f\nabla_{k}\mathring{R}_{ij}\\
&=&\mathring{R}_{ij}\nabla_{k}f\nabla_{i}\mathring{R}_{kj}\\
&=&\mathring{R}_{ij}(\nabla_{i}(\mathring{R}_{kj}\nabla_{k}f)-\nabla_{i}\nabla_{k}f\mathring{R}_{kj}).
\end{eqnarray*}
Thus, as we already know that $\mathring{Ric}(\nabla f)=\mathring{R}_{11}\nabla f$, then the above expression can be rewritten in the following way
\begin{eqnarray*}
\frac{1}{2}\nabla f(|\mathring{Ric}|^{2})&=&\mathring{R}_{ij}\nabla_{i}\mathring{R}_{11}\nabla_{j}f+\mathring{R}_{11}\mathring{R}_{ij}\nabla_{i}\nabla_{j}f\\
&&-fR_{ik}\mathring{R}_{ij}\mathring{R}_{jk}-(\Delta f+1)|\mathring{Ric}|^{2}\\
&=&\mathring{R}_{11}\langle\nabla\mathring{R}_{11} ,\nabla f\rangle+f\mathring{R}_{11}|\mathring{Ric}|^{2}-ftr(\mathring{Ric}^{3})\\
&&-\frac{R}{n}f|\mathring{Ric}|^{2}-(\Delta f+1)|\mathring{Ric}|^{2}.
\end{eqnarray*}
Therefore, it follows immediately from equations (\ref{normafRIC}) and (\ref{eqtrace}) that
\begin{eqnarray*}
\frac{1}{2}\nabla f(|\mathring{Ric}|^{2})&=&2f\mathring{R}_{11}|\mathring{Ric}|^{2}-\Delta f\mathring{R}_{11}^{2}+\mathring{R}_{11}f|\mathring{Ric}|^{2}\\
&&-ftr(\mathring{Ric}^{3})+\frac{(n-1)\Delta f+n}{n}|\mathring{Ric}|^{2}-(\Delta f+1)|\mathring{Ric}|^{2}\\
&=&2f\mathring{R}_{11}|\mathring{Ric}|^{2}-\Delta f\mathring{R}_{11}^{2}-ftr(\mathring{Ric}^{3})-\frac{\Delta f}{n}|\mathring{Ric}|^{2},
\end{eqnarray*}
as desired.
\end{proof}

Proceeding, we combine (\ref{normafRIC}) jointly with Lemma~\ref{nablafRic} in order to get the following expression
\begin{eqnarray*}
\frac{1}{2}\left(\frac{n-2}{n-1}\right)^{2}\nabla f(|W^{\nu}|^{2})&=&\frac{1}{2}\nabla f(|\mathring{Ric}|^{2})-\frac{n}{n-1}\mathring{R}_{11}\langle\nabla f,\nabla\mathring{R}_{11}\rangle\\
&=&2f\mathring{R}_{11}|\mathring{Ric}|^{2}-ftr(\mathring{Ric}^{3})-\Delta f \mathring{R}_{11}^{2}-\frac{\Delta f}{n}|\mathring{Ric}|^{2}\\
&&-\frac{n}{n-1}(f\mathring{R}_{11}|\mathring{Ric}|^{2}-\Delta f\mathring{R}_{11}^{2})\\
&=&\frac{n-2}{n-1}f\mathring{R}_{11}|\mathring{Ric}|^{2}-ftr(\mathring{Ric}^{3})+\frac{\Delta f}{n}\left(|\mathring{Ric}|^{2}-\frac{n}{n-1}\mathring{R}_{11}^{2}\right).
\end{eqnarray*}
This identity should be summarized in the following result.

\begin{lemma}\label{nablafW}
Let $(M,g,f)$ be a Miao-Tam critical metric with harmonic Weyl tensor. Then, for any regular point, we have
$$\frac{1}{2}\left(\frac{n-2}{n-1}\right)^{2}\nabla f(|W^{\nu}|^{2})=\frac{n-2}{n-1}f\mathring{R}_{11}|\mathring{Ric}|^{2}-ftr(\mathring{Ric}^{3})-\left(\frac{n-2}{n-1}\right)^{2}\frac{\Delta f}{n}|W^{\nu}|^{2}.$$
\end{lemma}

Finally, to prove the next result , we need the following Bochner type-formula established for $V$-static metrics which was proved by the first author and Ribeiro Jr. in \cite{balt17}. Here, for our purpose, we write this result just for Miao-Tam critical metrics with harmonic Weyl curvature.

\begin{proposition}\label{propA}
Let $(M^{n},\,g,\,f)$ be a Miao-Tam critical metric with harmonic Weyl tensor. Then we have:
\begin{eqnarray*}
\frac{1}{2}{\rm div}(f\nabla|Ric|^{2})&=&|\nabla Ric|^{2}f+ \frac{n}{n-1}|\mathring{Ric}|^{2}+\frac{2}{n-1}Rf|\mathring{Ric}|^{2}\\
&&+\frac{2n}{n-2}ftr(\mathring{Ric}^{3})-2fW_{ijkl}R_{ik}R_{jl}.
\end{eqnarray*}
\end{proposition}

Now, we are in position to present our last lemma in order to classify such critical metrics under harmonicity of the Weyl tensor. More precisely, we have the following result.

\begin{lemma}\label{divWnu}
Let $(M,g,f)$ be a Miao-Tam critical metric with harmonic Weyl tensor. Then, for any regular point, we have
$$\frac{n-2}{n-1}{\rm div}(|W^{\nu}|^{2}\nabla f)=-fW_{ijkl}R_{ik}R_{jl}.$$
\end{lemma}
\begin{proof}
From Lemma~\ref{nablafW} we achieve
\begin{eqnarray*}
\frac{n-2}{n-1}{\rm div}(|W^{\nu}|^{2}\nabla f)&=&\frac{n-2}{n-1}|W^{\nu}|^{2}\Delta f+\frac{n-2}{n-1}\nabla f(|W^{\nu}|^{2})\\
&=&\frac{(n-2)^{2}}{n(n-1)}|W^{\nu}|^{2}\Delta f+2f\mathring{R}_{11}|\mathring{Ric}|^{2}-\frac{2(n-1)}{n-2}ftr(\mathring{Ric}^{3}).
\end{eqnarray*}
Hence, substituting (\ref{auxnormaRIC}) in the above expression, we arrive at
\begin{eqnarray}\label{auxdivWnu}
\frac{n-2}{n-1}{\rm div}(|W^{\nu}|^{2}\nabla f)&=&\frac{n-1}{n}\Delta f|\mathring{Ric}|^{2}-\Delta f\mathring{R}_{11}^{2}+2f\mathring{R}_{11}|\mathring{Ric}|^{2}\nonumber\\
&&-\frac{2(n-1)}{n-2}ftr(\mathring{Ric}^{3}).
\end{eqnarray}

On the other hand, it follows from Proposition~\ref{propA} that
\begin{eqnarray}\label{auxnablaNRIC}
\frac{1}{2}\nabla f(|\mathring{Ric}|^{2})&=&-\frac{1}{2}f\Delta|\mathring{Ric}|^{2}+\frac{1}{2}div(f\nabla|\mathring{Ric}|^{2})\nonumber\\
&=&-\left(\frac{1}{2}\Delta|\mathring{Ric}|^{2}-|\nabla Ric|^{2}\right)f+\frac{n}{n-1}|\mathring{Ric}|^{2}+2B,
\end{eqnarray}
where $B$ is the function defined by
$$B=\frac{1}{n-1}fR|\mathring{Ric}|^{2}+\frac{n}{n-2}ftr(\mathring{Ric}^{3})-fW_{ijkl}R_{ik}R_{jl}.$$
A straightforward computation shows that the function $B$ defined above assume the following form
$$B=f(R_{ij}R_{ik}R_{jk}-R_{ik}R_{jl}R_{ijkl}),$$
see for instance, \cite[Lemma 4]{balt17}. Hence, by Lemma~\ref{deltaRICCI} jointly with the fact that $M$ has constant scalar curvature and harmonic Weyl tensor, we arrive at
\begin{equation*}
\frac{1}{2}f\Delta|Ric|^{2}=f|\nabla Ric|^{2}+B,
\end{equation*}
which combined with (\ref{auxnablaNRIC}) yields
\begin{eqnarray*}
\frac{1}{2}\nabla f(|\mathring{Ric}|^{2})&=&\frac{n}{n-1}|\mathring{Ric}|^{2}+B\\
&=&-\Delta f |\mathring{Ric}|^{2}+\frac{n}{n-2}ftr(\mathring{Ric}^{3})-fW_{ijkl}R_{ik}R_{jl}.
\end{eqnarray*}

Therefore, after we compare this expression with Lemma~\ref{nablafRic}, it is not difficult to see that
\begin{eqnarray*}
-fW_{ijkl}R_{ik}R_{jl}&=&\frac{n-1}{n}\Delta f|\mathring{Ric}|^{2}-\Delta f\mathring{R}_{11}^{2}+2f\mathring{R}_{11}|\mathring{Ric}|^{2}-\frac{2(n-1)}{n-2}ftr(\mathring{Ric}^{3})\\
&=&\frac{n-2}{n-1}{\rm div}(|W^{\nu}|^{2}\nabla f),
\end{eqnarray*}
where in the last equality we have used (\ref{auxdivWnu}). This finishes the proof of the lemma.

\end{proof}

\subsection{Proof of the Theorem~\ref{thmdivW}}
\begin{proof}
In what follows, consider $\Sigma_{0}$ as a connected component of the boundary $f^{-1}(0)$ and $\widetilde{M}_{0}$ be the connected component of the open set $\{x\in M;|\nabla f|>0\}$ such that its closure contains $\Sigma_{0}$ and let $M_{0}=\widetilde{M}_{0}\cup\Sigma_{0}.$ Next, for $\epsilon>0$ sufficiently small, we will consider the set $V_{\epsilon}=\{0\leq f\leq\epsilon\}\cap M_{0}.$

The main idea is to prove that $W(\cdot,\nabla f,\cdot,\nabla f)$ vanishes on $V_{\epsilon}$ and, as we will see, this is sufficient in order to conclude our theorem. In fact, combining Lemma~\ref{nablafW} and \ref{divWnu} we deduce
\begin{eqnarray*}
\frac{n-2}{n-1}fW_{ijkl}R_{ik}R_{jl}&=&-\left(\frac{n-2}{n-1}\right)^{2}(|W^{\nu}|^{2}\Delta f+\nabla f(|W^{\nu}|^{2}))\\
&=&-\left(\frac{n-2}{n-1}\right)^{2}|W^{\nu}|^{2}\Delta f-\frac{2(n-2)}{n-1}f\mathring{R}_{11}|\mathring{Ric}|^{2}\\
&&+2ftr(\mathring{Ric}^{3})+2\left(\frac{n-2}{n-1}\right)^{2}\frac{\Delta f}{n}|W^{\nu}|^{2}\\
&=&\frac{1}{n}\left(\frac{n-2}{n-1}\right)^{3}|W^{\nu}|^{2}(fR+n)-\frac{2(n-2)}{n-1}f\mathring{R}_{11}|\mathring{Ric}|^{2}\\
&&+2ftr(\mathring{Ric}^{3}).
\end{eqnarray*}
So, since $f|_{\Sigma_{0}}=0$, it is immediate to check that $|W^{\nu}|^{2}\equiv0$ on $\Sigma_{0}.$ In addition, still restricted to the boundary $\Sigma_{0},$ we may use the equality (\ref{auxRICCI}), to infer
$$\mathring{R}_{ij}=-\frac{\mathring{R}_{11}}{n-1}g_{ij},$$
for all $2\leq i,j\leq n$, and this allows us to conclude, jointly with Lemma~\ref{autoRIC},  that
\begin{eqnarray*}
W_{ijkl}R_{ik}R_{jl}=\sum_{i,j\neq1}W_{ijkl}R_{ik}R_{jl}=\left(\frac{R}{n}-\frac{\mathring{R}_{11}}{n-1}\right)W_{ijil}R_{jl}=0,
\end{eqnarray*}
where in the last step we have used that Weyl tensor has trace-free in any two indices.

Next, take $\zeta>0$ arbitrary and after a straightforward computation, we get
$$div((\zeta+|W^{\nu}|^{2})\nabla f)\Big|_{\Sigma_{0}}=-\frac{n\zeta}{n-1}<0.$$
Thus, by continuity of the function $div((\zeta+|W^{\nu}|^{2})\nabla f)$ there exist $\epsilon_{0}>0$ (possibly smaller than $\epsilon$) such that
\begin{equation}\label{inqdivzeta}
div((\zeta+|W^{\nu}|^{2})\nabla f)<0,
\end{equation}
on $V_{\epsilon_{0}}=\{0\leq f\leq\epsilon_{0}\}\cap M_{0}.$ For simplicity, we will also denote this new neighborhood by $V_{\epsilon}$. Next, it is not difficult to see, using  (\ref{inqdivzeta}), that the following inequality is true
\begin{eqnarray*}
div(|W^{\nu}|^{2}\nabla f)\Big|_{V_{\epsilon}}&=&-\zeta\Delta f+div((\zeta+|W^{\nu}|^{2})\nabla f)\\
&<&\frac{\zeta R}{n-1}f+\frac{\zeta n}{n-1}.
\end{eqnarray*}
Thus, when $\zeta\rightarrow0$, we conclude that
$$div(|W^{\nu}|^{2}\nabla f)\Big|_{V_{\epsilon}}\leq0.$$

In the sequel, using Lemma~\ref{nablafW} again jointly with Stokes' theorem, we may deduce
\begin{eqnarray*}
0&\leq&\int_{V_{\epsilon}}fW_{ijkl}R_{ik}R_{jl}=-\frac{n-2}{n-1}\int_{V_{\epsilon}} div(|W^{\nu}|^{2}\nabla f)\\
&=&-\frac{n-2}{n-1}\int_{\{f=0\}}|W^{\nu}|^{2}\langle\nabla f,\nu_{0}\rangle-\frac{n-2}{n-1}\int_{\{f=\epsilon\}}|W^{\nu}|^{2}\langle\nabla f,\nu_{\epsilon}\rangle\\
&=&-\frac{n-2}{n-1}\int_{\{f=\epsilon\}}|W^{\nu}|^{2}|\nabla f|\leq0,
\end{eqnarray*}
where $\nu_{0}=-\frac{\nabla f}{|\nabla f|}$ and $\nu_{\epsilon}=\frac{\nabla f}{|\nabla f|}$ denotes the outward unit normal to $\Sigma_{0}=\{f=0\}$ and $\{f=\epsilon\},$ respectively.
Therefore,
$$div(|W^{\nu}|^{2}\nabla f)=-\frac{n-1}{n-2}fW_{ijkl}R_{ik}R_{jl}=0,$$
on $V_{\epsilon},$ and
$$W^{\nu}|_{\{f=\epsilon\}}\equiv0.$$

Hence, upon taking the Stokes' theorem combined with the two above informations, we have that
\begin{eqnarray*}
0&=&\int_{V_{\epsilon}}fdiv(|W^{\nu}|^{2}\nabla f)dM_{g}\\
&=&\int_{V_{\epsilon}}div(f|W^{\nu}|^{2}\nabla f)dM_{g}-\int_{V_{\epsilon}}|W^{\nu}|^{2}|\nabla f|^{2}\\
&=&\int_{\{f=\epsilon\}}f|W^{\nu}|^{2}|\nabla f|d\sigma-\int_{V_{\epsilon}}|W^{\nu}|^{2}|\nabla f|^{2}\\
&=&-\int_{V_{\epsilon}}|W^{\nu}|^{2}|\nabla f|^{2},
\end{eqnarray*}
which implies that $W^{\nu}\equiv0$ on $V_{\epsilon}.$

In order to proceed, we recall that choosing appropriate coordinates, namely, harmonic coordinates, we conclude that the function $f$ and the metric $g$ are real analytic (for more details see, for instance, \cite[Proposition 2.8]{Corvino}). Consequently, since we already know that $W^{\nu}\equiv0$ in the open set $V_{\epsilon},$ we must obtain $W^{\nu}\equiv0$ on $M.$

Finally, using (\ref{CTW}) and the definition of tensor $T_{ijk}$ we get
\begin{eqnarray*}
|T_{ijk}|^{2}=-T_{ijk}W_{ijkl}\nabla_{l}f=-\frac{n-1}{n-2}W_{ijkl}\nabla_{l}f(\nabla_{j}fR_{ik}-\nabla_{i}fR_{jk})=0.
\end{eqnarray*}
Then $T\equiv0$ and we are in position to use Lemma~\ref{bachCotton} to conclude that $M$ is a Bach-flat manifold. So, the proof is completed.
\end{proof}

\end{document}